\documentclass[11pt]{article}
\usepackage{amsmath,amsthm,amscd,amssymb}
\usepackage{upref}
\usepackage{latexsym}
\usepackage{lscape}
\usepackage{graphicx}
\usepackage{color}
\usepackage[text={7.0in,9.5in},centering,includefoot,foot=0.6in]{geometry}

\makeatletter
\def\theequation{\@arabic\c@equation}

\newcommand{\bp}{{\mathbf{p}}}
\newcommand{\bq}{{\mathbf{q}}}

\newcommand{\bbR}{{\mathbb{R}}}
\newcommand{\R}{{\mathbb{R}}}

\newcommand{\cB}{{\mathcal B}}

\newcommand{\cD}{{\mathcal D}}

\newcommand{\cH}{{\mathcal H}}

\newcommand{\cM}{{\mathcal M}}

\newcommand{\cT}{{\mathcal T}}

\newcommand{\cV}{{\mathcal V}}

\newcommand{\cX}{{\mathcal X}}

\newcommand{\lb}{\label}

\newcommand{\sgn}{\text{\rm{sign }}}

\renewcommand{\ln}{\text{\rm ln}}

\newcommand{\dist}{\operatorname{dist}}

\newcommand{\mi}{\operatorname{Mas}}
\newcommand{\mo}{\operatorname{Mor}}

\numberwithin{equation}{section}

\renewcommand{\det}{\operatorname{det}}

\newcommand{\dom}{\operatorname{dom}}
\newcommand{\codim}{\operatorname{codim}}

\newcommand{\Sp}{\operatorname{Sp}}

\renewcommand{\ker}{\operatorname{ker}}
\newcommand{\diag}{\operatorname{diag}}

\theoremstyle{plain}
\newtheorem{theorem}{Theorem}[section]
\newtheorem{hypothesis}[theorem]{Hypothesis}
\newtheorem{lemma}[theorem]{Lemma}

\newtheorem{proposition}[theorem]{Proposition}

\theoremstyle{definition}
\newtheorem{definition}[theorem]{Definition}

\newtheorem{remark}[theorem]{Remark}

\begin{document}
\allowdisplaybreaks

\title{Instability of pulses in gradient reaction--diffusion systems: A symplectic approach.}

\author{M.\ Beck\thanks{mabeck@math.bu.edu}, G.\ Cox\thanks{gcox@mun.ca}, C.\ Jones\thanks{ckrtj@email.unc.edu}, Y.\ Latushkin\thanks{latushkiny@missouri.edu}, K.\ McQuighan\thanks{kmcquigh@bu.edu} and A.\ Sukhtayev\thanks{sukhtaa@miamioh.edu}}

\date{\today}

\maketitle

\begin{abstract}
In a scalar reaction--diffusion equation, it is known that the stability of a steady state can be determined from the Maslov index, a topological invariant that counts the state's critical points. In particular, this implies that pulse solutions are unstable. We extend this picture to pulses in reaction--diffusion systems with gradient nonlinearity. In particular, we associate a Maslov index to any asymptotically constant state, generalizing existing definitions of the Maslov index for homoclinic orbits. It is shown that this index equals the number of unstable eigenvalues for the linearized evolution equation. Finally, we use a symmetry argument to show that any pulse solution must have nonzero Maslov index, and hence be unstable.
\end{abstract}

 
\section{Introduction}\lb{s:intro}

Consider an $n \times n$ system of reaction--diffusion equations
\[
	u_t = D u_{xx} + G(u), \quad u\in\R^n,
\]
where $G\colon \bbR^n\to\bbR^n$ and $D$ is a positive diagonal matrix. We are interested in relating the stability of a steady state $\varphi^*$ to its underlying geometric structure. Recall that $\varphi^*$ is said to be spectrally unstable if the linear operator
\[
	L = D \frac{d^2}{dx^2} + \nabla G(\varphi^*(x))
\]
has any spectrum in the open right half-plane. Note that for a reaction--diffusion equation, the existence of unstable spectrum is enough to guarantee nonlinear instability; see \cite{SS98}.


 Since $L \varphi^*_x = 0$, we know that 0 is an eigenvalue of $L$ with eigenfunction $\varphi^*_x$, provided the state has suitable decay at $\pm\infty$. Thus in the scalar case $n=1$, Sturm--Liouville theory tells us that the number of unstable eigenvalues equals the number of zeros of $\varphi^*_x$. In particular, if $\varphi^*$ is a pulse, meaning $\varphi^*(x) \to 0$ as $|x| \to \infty$, it must have at least one critical point, and hence be unstable; c.f. \cite[\S 2.3]{KP}. However, when $n>1$, the results of Sturm--Liouville theory no longer apply, and it is not as easy to determine the stability of a pulse solution.

The question of pulse stability arose initially from a discussion of Turing patterns, which are spatially periodic patterns that arise in reaction--diffusion systems when a homogeneous background state becomes unstable in the presence of diffusion. In all contexts in which such patterns are physically observable, the diffusion coefficients are not all equal (and in fact are quite different in size). It has been conjectured that this is a necessary condition for the stability, and hence observability, of Turing patterns. In many situations, the stability of a spatially periodic pattern can be related to the stability of a nearby pulse. Thus, if one could prove that pulse solutions are necessarily unstable when the diffusion coefficients are all equal, it would provide a route to resolving the Turing pattern conjecture for periodic solutions.

In this paper we establish the instability of pulse solutions when $n>1$, under the assumption that the nonlinearity satisfies $G = \nabla F$ for some $F$. In this case the linearized operator $L$ is selfadjoint. This is not a setting in which Turning patterns arise, as the standard mechanism for their creation---a stable, homogeneous state destabilized by the addition of diffusion---does not occur in gradient systems \cite{T52}. Thus our result does not address the original Turing pattern conjecture, but we feel it is interesting in its own right. Our proof exploits the gradient structure of $G$ by recasting the eigenvalue problem for $L$ as a first-order Hamiltonian system. At present we do not know how to extend this analysis to an arbitrary (i.e. nongradient) nonlinearity $G$.

The starting point for our proof is Arnold's symplectic interpretation of Sturm--Liouville theory \cite{Arn85}. Arnold observed that for a system of differential equations, a useful generalization of the notion of ``oscillation" (or number of zeros) is given by the Maslov index, a topological invariant counting signed intersections of Lagrangian planes in a symplectic vector space.

Arnold's analysis is local in space. That is, he uses the symplectic picture to formulate and prove comparison-type results about oscillations and the interlacing of zeros for Hamiltonian systems. In the first part of the paper, we develop a global version of this picture, in which the number of unstable eigenvalues for a differential operator on a half-line $(-\infty,L]$, or the full line $\bbR$, is equated to the Maslov index of a family of Lagrangian planes. A similar analysis has been carried out in \cite{CH07,HLS2,HP17}. However, in those papers the Maslov index is defined using a different family of Lagrangian planes than is needed for our current application, so those results cannot be applied directly, and we must develop a new (though closely related) framework for the problem at hand. It is worth pointing out that these results hold for any value of $n$ (i.e. for a single equation or a system of any dimension).

Having related the number of unstable eigenvalues to the Maslov index, our next task is to compute the latter quantity. This is where the difference between the system and scalar cases becomes apparent. An explicit computation of the Maslov index requires $n$ linearly independent solutions to the equation $Lu = 0$. The derivative $\varphi^*_x$ of the steady state is always such a solution. Thus, when $n=1$, $\varphi^*$ contains all of the information needed to evaluate the Maslov index and determine whether or not $L$ has any unstable spectrum. This simple fact is the basis for the Sturm oscillation theorem.

When $n>1$, we cannot compute the Maslov index without knowing the additional $n-1$ solutions. However, in the case of a pulse, we can use a geometric argument to give a lower bound on the Maslov index, thus proving that at least one unstable eigenvalue exists. It is here that the advantages of our definition of the Maslov index (and in particular our choice of reference plane) become apparent. In particular, our definition is such that a zero in the derivative of the pulse will generate an unstable eigenvalue. This property is not enjoyed by previous definitions of the Maslov index for homoclinic orbits, which measure intersections between stable and unstable subspaces, much like the Evans function.

For the last few decades, the Evans function has been the primary tool for determining the stability of coherent structures in a range of partial differential equations, including reaction--diffusion equations but also many others, such as nonlinear Schr\"odinger equations and KdV-type equations \cite{S02}. However, the Evans function has not been shown to be useful for problems in multiple spatial dimensions. Recently, a series of papers have focused on developing the Maslov index as an alternative tool for analyzing stability in the multidimensional setting. Although that theory has seen much important progress \cite{CJLS, CJM1, CJM2, DJ11}, to our knowledge there is only one practical application of the Maslov index to determine stability in a setting where the Evans function does not also apply. This is in \cite{CH14}, where the Maslov index is used to prove that the standing pulse solution of the FitzHugh--Nagumo equation, with diffusion in both variables, is unstable, provided certain assumptions on the model parameters are satisfied. The analysis in \cite{CH14} is particular to the FitzHugh--Nagumo model, and relies heavily on the activator--inhibitor structure in evaluating the Maslov index. The results in the current paper are valid for a general class of equations, only requiring a certain generic assumption, and so we feel that this is a significant contribution towards demonstrating the utility of the Maslov index in a setting where the Evans function has not been used to obtain the same result. 

There is another context in which it was shown that pulse solutions are unstable for a class of systems of PDEs, namely the viscous conservation laws studied in \cite[Equation (4.1)]{GZ98}. In those systems of two equations, it was shown that homoclinic orbits, which correspond to undercompressive shocks, are necessarily unstable. Their proof involves a necessary condition for stability that is derived using the Evans function, and it is not clear how to utilize that method for the systems we consider here.
Further references on the Maslov index and the stability of pulses can be found in \cite{CH14}.


{\bf Notation:} We let $\langle\cdot\,,\cdot\rangle$ denote the real Euclidean scalar product of vectors, and let $\top$ denote transposition. For an $x$-dependent $n$-dimensional matrix $A\colon \R \to \R^{n\times n}$ we define the norm
$$
\| A \|_\infty = \sup_{x\in\R}\sup_{\substack{v\in\R^n\\\|v\|=1}} | A(x)v |.
$$
When $a=(a_i)_{i=1}^n\in\bbR^n$ and $b=(b_j)_{j=1}^m\in\bbR^m$ are $(n\times 1)$ and $(m\times 1)$ column vectors,  we use the notation $(a,b)^\top$ for the $(n+m)\times 1$ column vector with the entries $a_1,\dots,a_n,b_1,\dots,b_m$ (avoiding the use of $(a^\top,b^\top)^\top$). We denote by $\cB(\cX)$ the set of linear bounded operators on a Hilbert space $\cX$ and by $\Sp(T)=\Sp(T; \cX)$ the spectrum of an operator on $\cX$.

 
\subsection{Outline of main results and structure of paper} \lb{s1}

We describe a symplectic approach to counting negative eigenvalues of a second order differential operator on the real line, that is, the values $\lambda < 0$ for which there exists a nontrivial solution to the eigenvalue problem
\begin{equation}\label{ShEq}
	Hu := -Du'' + V(x)u = \lambda u, \quad D=\diag\{d_i\}>0,\quad u\in\bbR^n, \; x\in\R
\end{equation}
where 
\begin{equation}\lb{defD}
	\dom(H)= H^2(\R;\R^n)
\end{equation}
and $V(x)\in\R^{n\times n}$ satisfies the following hypotheses.
\begin{hypothesis} \lb{h1}
\leavevmode
\renewcommand{\labelenumi}{{(H\arabic{enumi})}}
\begin{enumerate}\item  The potential $V\in C(\bbR,\bbR^{n\times n})$ takes values in the set of symmetric matrices with real entries.
\item The limits $\lim_{x\to\pm\infty}V(x)=V_\pm$ exist, and are positive-definite matrices, i.e. $\Sp(V_\pm)>0$.
\item The functions $x\mapsto \big(V(x)-V_\pm\big)$ are in $L^{1}(\bbR_\pm;\R^{n\times n})$.
\end{enumerate}
\end{hypothesis}

Hypothesis (H1) ensures that the spectrum for \eqref{ShEq} is real. Hypothesis (H2) is a necessary and sufficient condition to ensure that the essential spectrum of $H$ is strictly positive, i.e., $k^2 D+V_{\pm}>0$ for every $k\in\bbR$ if and only if $V_{\pm}>0$. Finally, Hypothesis (H3) is important in applying the theory of exponential dichotomies to the first-order system of equations \eqref{ShEqp} below, which is equivalent to the eigenvalue problem \eqref{ShEq}.
In particular, this hypothesis ensures that the existence of an exponential dichotomy of the constant coefficient asymptotic system implies the existence of an exponential dichotomy of \eqref{ShEqp} and the continuity of the respective dichotomy projections, cf.\ Lemma \ref{lem:dich}.
\begin{remark}
The results from the first four sections can be extended to the case $Hu = -(D(x)u')' + V(x)u$, where $D\in C(\bbR,\bbR^{n\times n})$ takes values in the set of symmetric matrices with real entries, there is $\alpha>0$ such that $D(x)\geq\alpha I>0$ for all $x\in\bbR$, and the limits $\lim_{x\to\pm\infty}D(x)=D_\pm$ exist (these restrictions on $D$ are described in \cite{BL99}).
\end{remark}
\begin{remark}
The results can also be extended to the case where the essential spectrum touches the imaginary axis, meaning that one or both of $V_\pm$ is not hyperbolic. In that case, we could instead consider the operator $H+C$ for some small positive constant $C$ so that $V_\pm + C$ is positive. This would allow us to use the below results to compute the Morse index of $H+C$, which equals the Morse index of $H$ for sufficiently small $C$. The case where one or both of $V_\pm < 0$ is not of interest, because in this case the state is already unstable due to the essential spectrum.
\end{remark}

Operators of the form \eqref{ShEq} arise, for example, when determining the stability of pulse or front solutions to a reaction--diffusion system with gradient nonlinearity:
\begin{align*}
u_t =D \Delta u+G(u), \quad u\in\R^n,
\end{align*}
where $G(u) = \nabla F(u)$ for some scalar function $F$; see the example in Section~\ref{ss:example}. If $\varphi_*$ is the solution about which one linearizes, then the spectrum of $\mathcal{L} =D \Delta + \nabla^2F(\varphi_*(x))$ determines the linear stability of $\varphi_*$. By considering instead the operator $H = -\mathcal{L}$, we have that the Morse index of $H$, which is defined to be the number of negative eigenvalues, equals the number of positive, and hence unstable, eigenvalues for $\mathcal L$. 

Our strategy is to first consider the eigenvalue problem on the half line
\begin{equation}\label{ShEqL}
	H_Lu := -Du'' + V(x)u = \lambda u, \quad u\in\bbR^n, \quad x\in(-\infty,L]
\end{equation}
where 
\begin{equation}\lb{defLD}
	\dom(H_L)=\Big\{ u\in H^2((-\infty,L];\R^n) \, \Big| \, u(L)=0\Big\}
\end{equation}
and $L \in \bbR$ is fixed. 
We then extend our results to the full line in Section~\ref{ss:line}. 

Setting
\begin{equation}\label{Defpq}\begin{split}
p_1 &:= u \in \R^{n}\\
p_2 & := Du' \in \R^{n}
\end{split}\end{equation}
and ${\bf p} := (p_1,p_2)^\top \in \R^{2n}$, we can write \eqref{ShEqL} as
\begin{equation} \label{ShEqp}
\bp' = A(x,\lambda) \bp, \quad A(x,\lambda)= \begin{pmatrix} 0_{n} & D^{-1} \\ -\lambda I_{n} +  V(x) & 0_{n} \end{pmatrix}
\end{equation}
where $I_n$ and $0_n$ denote the $n\times n$ identity and zero matrices, respectively. Then the Dirichlet boundary condition at $x=L$ corresponds to $\bp(L)\in\mathcal D$, where $\mathcal D$ is the {\em Dirichlet subspace} defined by
\begin{equation}\label{DefX1}
\mathcal D = \{ (p_1,p_2)^\top\in \R^{2n}\,\big|\, p_1=0 \}.
\end{equation}

We then define a trace map $\Phi_s$ that maps a solution $\bp$ of the system \eqref{ShEqp} to its boundary value $\bp(s)\in\R^{2n}$. Let $Y_\lambda$ denote the set of solutions of \eqref{ShEqp} that decay at $-\infty$. This leads to the critical observation (see Proposition \ref{cor:mult} below) that $\lambda$ is an eigenvalue of \eqref{ShEqL} on $(-\infty,s]$ if and only if the subspace $\Phi_s(Y_\lambda)$ intersects $\mathcal D$ nontrivially.

In what follows, we will be working with infinite rectangles $[-\lambda_\infty,0] \times [-\infty,L]$, where $\lambda_\infty$, $L\in\R_+$ are fixed, and we have compactified $(-\infty,L]$ as follows. We first identify $[-1,1]$ and $[-\infty,+\infty]$ via the mutually inverse functions
\begin{align}\label{DFNSS}
s(\sigma) = \frac12\ln\left(\frac{1+\sigma}{1-\sigma}\right),\quad \sigma(s) = \tanh(s), \quad \sigma\in[-1,1],\; s\in[-\infty,+\infty],
\end{align}
with $s(\pm1) = \pm\infty$ and $\sigma(\pm\infty) = \pm1$. Then there is a unique topology on $[-\infty,+\infty]$ so that both $s$ and $\sigma$ are continuous. Thus, by construction, the set $[-\infty,L]$ is homeomorphic to $[-1,\sigma(L)]$.


We will denote the boundary of $[-\lambda_\infty,0]\times[-\infty,L]$ by $\Gamma_L=\Gamma_{1,L}\cup\Gamma_{2,L}\cup\Gamma_{3,L}\cup\Gamma_{4,L}$, where $\Gamma_{1,L} = [-\lambda_\infty,0] \times \{-\infty\}$,  $\Gamma_{2,L} = \{0\} \times [-\infty,L]$, $\Gamma_{3,L} = [-\lambda_\infty,0] \times \{L\}$ and $\Gamma_{4,L}= \{-\lambda_\infty\} \times [-\infty,L]$, and each curve is oriented as shown in Figure~\ref{F1}. With a slight abuse of notation we also let $\Gamma_{j,L}$ denote the images of these line segments under the map $\sigma(\cdot)$ from \eqref{DFNSS}.

We extend the definition of $\Phi_s(Y_\lambda)$ from $s\in(-\infty,L]$ to $s\in[-\infty,L]$ by introducing the planes
$$
	\Phi_{-\infty}(Y_\lambda):=\mathbb E_-^u(-\infty,\lambda),
$$
where $\mathbb E_-^u(-\infty,\lambda)$ denotes the spectral subspace corresponding to the eigenvalues of the matrix
\begin{align*}
A_-(\lambda) := \begin{pmatrix}
0_n & D^{-1}\\-\lambda I_n+V_-&0_n
\end{pmatrix} = \lim_{x\to-\infty} A(x,\lambda),
\end{align*}
with positive real parts.

There is a natural symplectic structure on $\R^{2n}$ such that the subspaces $\Phi_s(Y_\lambda)$ and $\mathcal D$ are Lagrangian (see Theorem \ref{th:lagrange}). This symplectic structure allows us to define the Maslov index, Definition~\ref{def:maslov}, of $\Phi_s(Y_\lambda)$ with respect to $\mathcal D$, along the paths $\Gamma_{j,L}$ in the $\lambda$--$s$ plane.

We define a \emph{crossing} to be a point $(\lambda_*,s_*)$ for which $\Phi_s(Y_\lambda)$ has a non-trivial intersection with $\mathcal D$. The Maslov index of $\Gamma_{j,L}$ (with respect to $\mathcal D$) can be viewed as the number of crossings along $\Gamma_{j,L}$, counted with sign and multiplicity.
We will prove in Lemma~\ref{lem:dich} that the map
$$
	(\lambda,\sigma)\mapsto\Phi_{s(\sigma)}(Y_\lambda)
$$
from $[-\lambda_\infty,0] \times [-1,\ell]$ into the set of $n$-dimensional subspaces of $\R^{2n}$ is continuous, so its Maslov index is well defined. Our goal is to relate this Maslov index to the Morse index of \eqref{ShEqL}. Our strategy, as depicted in Figure~\ref{F1}, is as follows.

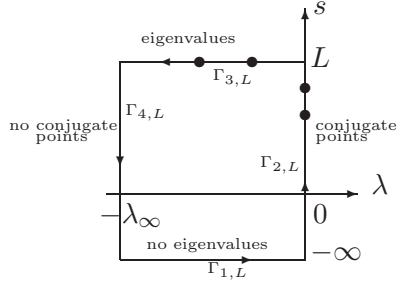
\begin{figure}
\centering
\begin{picture}(100,100)(-20,0)
\put(2,19){$-\lambda_{\infty}$}
\put(10,8){\line(0,1){4}}
\put(80,5){\vector(0,1){95}}
\put(10,5){\line(0,1){75}}
\put(10,80){\vector(0,-1){40}}
\put(5,30){\vector(1,0){95}}
\put(63,40){\text{\tiny $\Gamma_{2,L}$}}
\put(12,60){\text{\tiny $\Gamma_{4,L}$}}
\put(-32,54){\text{\tiny no conjugate}}
\put(-22,50){\text{\tiny points}}
\put(84,54){\text{\tiny conjugate}}
\put(84,50){\text{\tiny points}}
\put(45,73){\text{\tiny $\Gamma_{3,L}$}}
\put(43,0){\text{\tiny $\Gamma_{1,L}$}}
\put(105,30){$\lambda$}
\put(83,98){$s$}
\put(80,5){\vector(0,1){30}}
\put(83,19){$0$}
\put(10,5){\line(1,0){70}}
\put(10,5){\vector(1,0){50}}
\put(10,80){\line(1,0){70}}
\put(80,80){\vector(-1,0){55}}
\put(82,78){$L$}
\put(82,5){$-\infty $}
\put(80,60){\circle*{4}}
\put(80,70){\circle*{4}}
\put(40,80){\circle*{4}}
\put(60,80){\circle*{4}}
\put(18,87){{\tiny \text{eigenvalues}}}
\put(20,9){{\tiny \text{no eigenvalues}}}
\end{picture}
\caption{Illustrating the proof of Theorem \ref{th:mas2}: When $\lambda_\infty$ is large enough, there are no crossings on $\Gamma_{1,L}$ and $\Gamma_{4,L}$, and the Morse index equals the number of crossings on $\Gamma_{3,L}$. By homotopy invariance, the Morse index is equal to the number of crossings on $\Gamma_{2,L}$; these are precisely the conjugate points in $(-\infty,L)$.}\label{F1}
\end{figure} 

As above, we let $\Gamma_L=\Gamma_{1,L}\cup\Gamma_{2,L}\cup\Gamma_{3,L}\cup\Gamma_{4,L}$, where $\Gamma_L$ depends on the choices of both $L$ and $\lambda_\infty$. A homotopy argument implies that the Maslov index $\mi(\Gamma_L, \mathcal D)$ of the closed curve $\Gamma_L$ is equal to zero. By general properties of the Maslov index one has $\mi(\Gamma_L, \mathcal D)=\sum_{j=1}^4\mi(\Gamma_{j,L}, \mathcal D)$, using the orientation shown in Figure~\ref{F1}.

In Lemma \ref{lem:llss} we show that there are no crossings along $\Gamma_{1,L}$ or $\Gamma_{4,L}$ provided $\lambda_\infty$ is chosen large enough. This follows from the structure of $H_L$ and the fact that its spectrum is bounded from below uniformly in $L$.

Crossings along $\Gamma_{3,L}$ (when $s=L$ and $\lambda\in[-\lambda_\infty,0]$) correspond to eigenvalues of the operator $H_L$ defined in \eqref{ShEqL}, hence $\mo(H_L)$ equals the number of crossings on $\Gamma_{3,L}$, counting multiplicities. On the other hand, a local computation shows that all crossings along $\Gamma_{3,L}$ have the same sign; see Lemma \ref{lem:cross}. This important monotonicity property implies that $|\mi(\Gamma_{3,L},\cD))|$ is equal to the total number of crossings, and hence is equal to the Morse index of $H_L$.

Finally, $\mi(\Gamma_{2,L}, \mathcal D)$ is the Maslov index of the boundary value problem \eqref{ShEqL} with $\lambda=0$.  Another local computation shows that all crossings along $\Gamma_{2,L}$ have the same sign; see Lemma \ref{lem:scross}. This means, in particular, that $\mi(\Gamma_{2,L},\mathcal D)$ is equal to the number of crossings along $\Gamma_{2,L}$, counted with multiplicity.

Combining the above results with the fact that $\mi(\Gamma_L, \mathcal D)$ vanishes, we arrive at the desired formula
\begin{equation}\label{MoMiform}
	\mo(H_L) = \text{\# crossings, with multiplicity, in } (-\infty,L)
\end{equation}
relating the Maslov index of the boundary value problem and the Morse index of the corresponding differential operator. This result is summarized in Theorem \ref{th:mas2}. We then show in Theorem~\ref{main-line-thm} that the Maslov index of $\Gamma_{2,L}$ is independent of $L$ for all $L$ large enough, and thus can be viewed as the Maslov index for \eqref{ShEq} posed on the whole line. 

The paper is organized as follows. After a brief introduction to the Maslov index, in Section \ref{secsa} we introduce an appropriate symplectic structure and relate the crossings of the path $\Phi_s(Y_\lambda)$ to the eigenvalues of differential operators. In Section \ref{mmi} we prove monotonicity of the crossings and provide the ingredients for the main results of the paper. In Section~\ref{ss:line} we extend our results from Section~\ref{mmi} to the operator \eqref{ShEq} on the whole line. Finally, in Section~\ref{ss:example} we apply our results to prove the instability of pulse solutions in reaction--diffusion systems with gradient nonlinearity. 


\section{The Maslov index: A symplectic approach to counting eigenvalues}\label{secsa}
We begin by recalling some notions regarding symplectic structures and the Maslov index; for a detailed exposition see \cite{arnold67,dG, rs93}. For a brief but extremely informative account see \cite{FJN}.  Many of these ideas have been extended to infinite-dimensional settings; see, for instance, \cite{F}. 

A skew-symmetric, non-degenerate bilinear form $\omega$ on $\R^{2n}$ is said to be symplectic.
For every such $\omega$,
there exists a unique skew-symmetric matrix $\Omega$ so that  $$\omega(v_1,v_2)=\langle  v_1, \Omega v_2\rangle,\quad v_1,\, v_2\in\R^{2n}.$$ A Lagrangian plane $V$ is an $n$-dimensional subspace in $\R^{2n}$ such that $\omega(v_1,v_2)=0$ for all $v_1,v_2\in V$. The set of all Lagrangian planes in $\R^{2n}$ is denoted by $\Lambda(n)$. Note that $\Lambda(n)$ depends on the choice of $\omega$.

Let $\cT(V)$ denote the train of a fixed Lagrangian plane $V\in\Lambda(n)$, that is, the set of all Lagrangian planes whose intersection with $V$ is nontrivial. Obviously $\cT(V)=\cup_{k=1}^n\cT_k(V)$, where $\cT_k(V)=\big\{W\in\Lambda(n)\,\big|\, \dim(V\cap W)=k\big\}$, the set of all Lagrangian planes whose intersection with $V$ is $k$ dimensional. Each set $\cT_k(V)$ is a submanifold of $\Lambda(n)$ of codimension $k(k+1)/2$; in particular, $\codim\cT_1(V)=1$. Moreover, $\cT_1(V)$ is two-sidedly embedded in $\Lambda(n)$, in the sense that there is a nowhere-vanishing normal vector field on $\cT_1(V)$. This vector field defines a canonical orientation, hence one can speak about the positive and negative sides of $\cT_1(V)$. A continuous curve $\Phi$ in $\Lambda(n)$ can be perturbed so that it only intersects the train in the $\cT_1(V)$ component, and does so transversely. Then $\mi(\Phi,V)$, the {\em Maslov index of $\Phi$ with respect to $V$}, is defined to be the signed number of intersections of the curve $\Phi$ with $\cT_1(V)$. This idea is made precise in \cite{arnold67}.


For our purposes it suffices to define the Maslov index for regular, piecewise $C^1$ paths, following \cite{rs93}. Let $\Phi\colon [a,b]\to\Lambda(n)$ be a continuously differentiable path and fix a particular Lagrangian plane $V$. A {\em crossing} is a point $t_0\in[a,b]$ with $\Phi(t_0) \in \cT(V)$, i.e. $\Phi(t_0) \cap V \neq \{0\}$.  Let $t_0\in[a,b]$ be a crossing, and let $W$ be a subspace in $\R^{2n}$ transversal to $\Phi(t_0)$; generically $W$ can be chosen to be $V^\bot$. Then $W$ is transversal to $\Phi(t)$ for all $t\in[t_0-\varepsilon,t_0+\varepsilon]$ for $\varepsilon>0$ small enough.
Thus, there exists a family of matrices $\phi(t)$, viewed as operators from $\Phi(t_0)$ into $W$, so that $\Phi(t)$ is the graph of $\phi(t)$ for $|t-t_0|\le\varepsilon$. The bilinear form $Q_\cM$ defined by
\begin{equation}\label{QMF}
Q_\cM({\rm v}, {\rm w})=\frac{d}{dt}\omega({\rm v}, \phi(t){\rm w})\Big|_{t=t_0} \,\text{ for }\, {\rm v}, {\rm w}\in \Phi(t_0)\cap V,
\end{equation} is called the {\em crossing form}. We denote by $n_+(Q_\cM)$ and $n_-(Q_\cM)$ the number of positive and negative eigenvalues of $Q_\cM$, counted with multiplicity, so that $\sgn Q_\cM = n_+(Q_\cM) - n_-(Q_\cM)$ is its signature. Since $Q_\cM$ is a symmetric matrix its eigenvalues are purely real; thus the notion of positive and negative eigenvalues is well defined. A crossing is called {\em regular} if the crossing form is nondegenerate. 

\begin{definition}\lb{def:maslov}
If $s_0$ is the only regular crossing of the path $\Phi$ with $\cT(V)$ in a segment $[a_0,b_0]$, then the {\em Maslov index} $\mi(\Phi\big|_{[a_0,b_0]}, V)$ is defined as
\begin{equation}\label{DefMI}\begin{split}
\mi(\Phi\big|_{[a_0,b_0]}, V) &:= 
\Bigg\{\begin{array}{ll}
-n_-(Q_\cM)&\text{if } s_0=a_0\\
n_+(Q_\cM)-n_-(Q_\cM)&\text{if } s_0\in(a_0,b_0)\\
n_+(Q_\cM)&\text{if } s_0=b_0
\end{array}
\end{split}
\end{equation}
The {\em Maslov index} of any regular, piecewise $C^1$ path can be determined by computing the Maslov index on each segment and summing.
\end{definition}

The important features of the Maslov index for this work are summarized below.
\begin{theorem} \cite{rs93}
\begin{enumerate}
\item \label{it:cat} {\em (Additivity)} For $a < c < b$,
$$\mi(\Phi\big|_{[a,b]}, V)  = \mi(\Phi\big|_{[a,c]},V) +  \mi(\Phi\big|_{[c,b]},V).$$
\item \label{it:hom} {\em (Homotopy invariance)} Two paths $\Phi_0, \Phi_1 \colon [a,b] \to \Lambda(n)$, with $\Phi_0(a) = \Phi_1(a)$ and $\Phi_0(b) = \Phi_1(b)$, are homotopic with fixed endpoints if and only if they have the same Maslov index.
\end{enumerate}
\end{theorem}

\begin{remark}\label{rem:endpoints}
When defining the Maslov index for a curve with $\Phi(a) \neq \Phi(b)$, one needs to be careful about crossings at the endpoints, to ensure that the additivity property holds. That is, if $c \in (a,b)$ is a crossing, should it contribute towards $\mi(\Phi|_{[a,c]},V)$ or $\mi(\Phi|_{[c,b]},V)$? Our (arbitrary) convention is that the positive part of the crossing is assigned to $[a,c]$, whereas the negative part is assigned to $[c,b]$. We could just as well have chosen the opposite sign convention, or could have assigned $\frac12 \sgn Q_\cM$ to each segment. For a closed curve these are all equivalent, so our choice ultimately does not matter, but it will affect some intermediate results; see in particular Theorem \ref{th:mas2} and Remark \ref{rem:sign}.
\end{remark}

In our application to the spectral theory of differential operators, a special role is played by paths for which all crossings are sign definite (and have the same sign).

\begin{remark} \label{rem:multiplicity}
A crossing at $t_0$ is said to be {\em positive} if the crossing form is positive definite; in this case the local contribution to the Maslov index is just $n_+(Q_\cM) = \dim (\Phi(t_0) \cap V)$. A curve $\Phi \colon [a,b] \to \Lambda(n)$ is said to be positive if all of its crossings are positive, in which case
\[
	\mi(\Phi,V) = \sum_{a<t\leq b} \dim(\Phi(t) \cap V).
\]
Similarly, the Maslov index of a negative curve is given by
\[
	\mi(\Phi,V) = \sum_{a\leq t < b} \dim(\Phi(t) \cap V).
\]
Both sums are necessarily finite because regular crossings are isolated.
\end{remark}


 \subsection{Symplectic formulation of the eigenvalue problem}
We now return to the system \eqref{ShEqp}. We recall that $Y_\lambda$ denotes the ($n$-dimensional) space of solutions to the equation \eqref{ShEqp} that decay at $-\infty$, $\Phi_s$ denotes the trace map
\[
	\Phi_s\colon  \bp(\cdot) \mapsto \bp(s) \in \R^{2n}
\]
which maps a solution in $Y_\lambda$ to its value at $x=s$,
\[
	\mathcal D = \{ (p_1,p_2)^\top\in \R^{2n}\,\big|\, p_1=0 \}
\]
encodes the Dirichlet boundary conditions, and we have compactified $(-\infty,L]$ using the map $s(\sigma)$ from \eqref{DFNSS}.

The following lemma is needed to define (and compute) the Maslov index.

\begin{lemma}\label{lem:dich}
The map
$$
	(\lambda,\sigma)\mapsto\Phi_{s(\sigma)}(Y_\lambda)
$$
into the set of $n$-dimensional subspaces of $\R^{2n}$ is continuous on $[-\lambda_\infty,0] \times [-1,\ell]$ and $C^1$ on $[-\lambda_\infty,0] \times (-1,\ell]$; here $s(\sigma)$ is defined in \eqref{DFNSS} and $\lambda_\infty>0$ and $\ell\in(0,1)$ are fixed.
\end{lemma}
In what follows we will suppress the notation $s(\sigma)$ and will view $(\lambda,s)\mapsto\Phi_s(Y_\lambda)$ as a continuous map from $[-\lambda_\infty,0] \times [-\infty,L]$ into the set of $n$-dimensional subspaces of $\R^{2n}$.
\begin{proof}
We will use exponential dichotomies. By Hypothesis~\ref{h1} (H2), the essential spectrum of $H$ is strictly positive. Thus, $H-\lambda I$ is Fredholm for all $\lambda\in[-\lambda_\infty,0]$ and so the first order differential operator $\frac{\mathrm d} {\mathrm d x} - A(\cdot, \lambda)$ is Fredholm in $L^2(\R;\R^{2n})$; see, e.g. \cite{SS08} and \cite[Thm. 4.1]{GLS}. By Palmer's Theorem \cite{Pal1,Pal2}, see also \cite{B-AG}, the differential equation \eqref{ShEqp} has exponential dichotomies on $\R_+$ and $\R_-$. We will denote by $\mathbb E_-^u(x,\lambda)$ the dichotomy subspace of \eqref{ShEqp} for $x\in(-\infty,L]$, consisting of the values at $x$ of solutions of \eqref{ShEqp} which decay exponentially to zero at $-\infty$. That $\mathbb E_-^u(x,\lambda)$ is continuous for $x\in(-\infty,L]$ is true by definition. 

We note that the asymptotic (constant coefficient) equation $\bp'(s) = A_-(\lambda)\bp(s)$ is also exponentially dichotomic, that is, $\Sp(A_-(\lambda))\cap\mathrm i\mathbb R = \varnothing$, and we denote by $\mathbb E_-^u(-\infty,\lambda)$ the ($x$-independent) unstable subspace associated with the asymptotic system. Then the projections in $\R^{2n}$ onto $\mathbb E_-^u(x,\lambda)$ converge to the projection onto $\mathbb E^u_-(-\infty,\lambda)$ as $x\to-\infty$, see e.g. \cite{Cop} or \cite{DK}. Now the relation $\Phi_s(Y_\lambda) = \mathbb E_-^u(s,\lambda)$ for $s\in(-\infty,L]$ finishes the proof.
\end{proof}

\begin{theorem}\label{th:lagrange}
For all $s \in [-\infty,+\infty)$ and $\lambda \in (-\infty,0]$ the plane $\Phi_s(Y_\lambda)$ belongs to the space $\Lambda(n)$ of Lagrangian $n$-planes in $\R^{2n}$, with the Lagrangian structure $\omega(v_1,v_2)=\langle v_1,\Omega v_2\rangle$, where
\begin{equation}\label{eq:Omega}
	\Omega = \begin{pmatrix} 0_n & -I_n \\ I_n & 0_n \end{pmatrix}.
\end{equation}
\end{theorem}

\begin{proof}
We must prove that $\Phi_s(Y_\lambda)$ is $n$-dimensional, and $\left<v_1,\Omega v_2\right> = 0$ for any $v_1, v_2 \in \Phi_s(Y_\lambda)$. First consider $s\in(-\infty,L]$. For 
$v_1=\bp(s)$ and $v_2=\bq(s)$ we compute
\begin{align*}
\langle v_1,\Omega v_2\rangle&=\langle \bp(s),\Omega \bq(s)\rangle\\
&=\int_{-\infty}^s\frac{d}{dx}\Big(\langle \bp(x),\Omega \bq(x)\rangle\Big)\,dx\\
&=\int_{-\infty}^s\Big(\langle \bp'(x),\Omega \bq(x)\rangle+\langle \bp(x),\Omega \bq'(x)\rangle\Big)\,dx\\
&=\int_{-\infty}^s\Big(\langle A(x,\lambda)\bp(x),\Omega \bq(x)\rangle+\langle \bp(x),\Omega A(x,\lambda)\bq(x)\rangle\Big)\,dx\\
&=\int_{-\infty}^s\Big(-\langle \bp(x), \Omega A(x,\lambda)\bq(x)\rangle+\langle \bp(x),\Omega A(x,\lambda)\bq(x)\rangle\Big)\,dx=0,\end{align*}
where in the last line we used  $\Omega^\top=-\Omega$ and
\begin{equation}\label{symm}
	\Omega A(x,\lambda)=\big(\Omega A(x,\lambda)\big)^\top.\quad
\end{equation}
For $s=-\infty$, the notation $\Phi_{-\infty}(Y_\lambda)$ refers to $\mathbb E_-^u(-\infty,\lambda)$, the unstable subspace at $-\infty$. Let $v_1$, $v_2\in\mathbb E_-^u(-\infty,\lambda)$ and consider the evolution of $v_1$ and $v_2$ under the vector field
\[
	\bp' = A_-(\lambda)\bp, \quad A_-(\lambda)=\lim_{x\to-\infty}A(x,\lambda).
\]
Then by \eqref{symm} $$\frac{\mathrm d}{\mathrm d x} \langle \mathrm e^{xA_-(\lambda)}v_1, \Omega \mathrm e^{xA_-(\lambda)}v_2\rangle = 0.$$ Thus $\langle v_1,  \Omega v_2\rangle = \langle \Omega \mathrm e^{xA_-(\lambda)}v_1,\mathrm e^{xA_-(\lambda)}v_2\rangle = 0$ since $\mathrm e^{xA_-(\lambda)}v_{1,2}\to 0$ as $x\to-\infty$ by definition of the unstable subspace.

To complete the proof, we show that $\dim \Phi_s(Y_\lambda) = n$. By continuity, it suffices to prove the result at $s=-\infty$. Using
\[
	A_-(\lambda) = \begin{pmatrix}
	0_n & D^{-1} \\
	-\lambda I_n + V_- &0_n
	\end{pmatrix}
\]
we see that $\pm\nu \in \Sp(A_-(\lambda))$ if and only if $(-\lambda I_n + V_-)p = \nu^2 Dp$ for some $p \in \bbR^n$. Since $V_-$ and $D$ are positive, and $\lambda \leq 0$, we must have $\nu^2>0$, so $\nu$ is nonzero. This means $A_-(\lambda)$ has an equal number of positive and negative eigenvalues, hence $\dim \Phi_{-\infty}(Y_\lambda) = \dim \mathbb E_-^u(-\infty,\lambda) = n$ as claimed.
%
%
\end{proof}

It is also true that $\mathcal D\in\Lambda(n)$; this can be verified by a straightforward calculation.

\begin{definition}\label{def:cpev}
For a given $\lambda \in \bbR$, a point $s\in(-\infty,L]$ is called a {\em $\lambda$-conjugate point} of \eqref{ShEqp} if $\Phi_s(Y_\lambda) \cap \mathcal D \neq \{0\}$. In the special case $\lambda=0$, $s$ is simply called a \emph{conjugate point}.
\end{definition}

Thus $s$ is a $\lambda$-conjugate point if and only if there exists a nonzero solution ${\bf p} = (p_1,p_2)^\top$ to \eqref{ShEqp} that decays at $-\infty$ and satisfies the boundary condition $p_1(s) = 0$. This means \eqref{ShEqL} has a nonzero solution in $H^2((-\infty,s]; \bbR^n)$ with $u(s) =0$, hence $\lambda$ is an eigenvalue of $H_s$.

\begin{proposition}\label{cor:mult}
For any $\lambda\in\R$ and $s\in(-\infty,L]$ the following assertions are equivalent:
\begin{enumerate}
\item[$(i)$]\, $\lambda$ is an eigenvalue of $H_{s}$;

\item[$(ii)$]\, $s$ is a $\lambda$-conjugate point of \eqref{ShEqp}.
\end{enumerate}
Moreover, the multiplicity of the eigenvalue $\lambda$ is equal to the dimension of the subspace $\Phi_s(Y_\lambda)\cap \mathcal D$.
 \end{proposition}

As we will see in Lemma \ref{lem:llss}, for $\lambda_\infty$ large enough $H_s$ has no eigenvalues with $\lambda\le-\lambda_\infty$. Therefore, we may restrict $\lambda$ to $[-\lambda_\infty,0]$. For a fixed $\lambda_\infty $ we can view $\Phi_s$ as a continuous map from the rectangle $[-\lambda_\infty,0] \times [-\infty,L]$ to the space of Lagrangian planes $\Lambda(n)$. Let $\Gamma_L$ denote the boundary of the image of this map, and define $\Gamma_{1,L}=\big\{\Phi_{-\infty}(Y_\lambda)\big|\, \lambda\in[- \lambda_\infty,0]\big\}$, $\Gamma_{2,L}=\big\{\Phi_{s}(Y_0)\big|\, s\in[-\infty,L]\big\}$, $\Gamma_{3,L}=\big\{\Phi_L(Y_\lambda)\big|\, \lambda\in[0,-\lambda_\infty]\big\}$, and $\Gamma_{4,L}=\big\{\Phi_{s}(Y_{-\lambda_\infty})\big|\, s\in[-\infty, L]\big\}$, so that $\Gamma_L = \cup_{i}\Gamma_{i,L}$, as in Figure \ref{F1}.
Since $\Phi_s$ is continuous, $\Gamma_L$ is homotopic to a point, hence $\mi(\Gamma_L; \mathcal D) = 0$.
Letting $A_{i,L}$ denote the Maslov index of $\Gamma_{i,L}$ with respect to the Dirichlet subspace $\mathcal D$, we have
\begin{equation}\label{htpy}
	A_{1,L} + A_{2,L} + A_{3,L} + A_{4,L} = 0. 
\end{equation}




\begin{remark}\label{remSigned}
The Maslov index counts signed crossings and so, in general, $|A_{i,L}|$ is a lower bound for the total number of crossings. However, Lemmas~\ref{lem:cross} and \ref{lem:scross} below imply that the crossings along any given $\Gamma_{i,L}$ are of the same sign, which means $|A_{i,L}|$ is in fact equal to the number of crossings. 
\end{remark}

\subsection{No crossings on $\Gamma_{1,L}$ and $\Gamma_{4,L}$}
We now show that $A_{1,L}=0$ and, provided $\lambda_\infty$ is large enough, $A_{4,L}=0$. We use the following result.
\begin{theorem}\label{KatoTh}{\bf\cite[Theorem V.4.10]{Kato}}
Let $\cH$ be selfadjoint and $\cV\in\cB(\cX)$ be a symmetric operator on a Hilbert space $\cX$. Then $$\dist\big(\Sp(\cH+\cV),\Sp(\cH)\big)\le\|\cV\|_{\cB(\cX)}.$$
\end{theorem}

\begin{lemma}\label{lem:llss}
The following hold for all $L\in\R$:
\begin{itemize}
\item[(i)] \label{llsshyp2} $A_{1,L}=0$;
\item[(ii)] \label{llsshyp1} If $\lambda_\infty>\|V\|_\infty$, then  $A_{4,L}=0$.
\end{itemize}
\end{lemma}

\begin{remark}
In fact, the proof will show that there are no crossings on $\Gamma_{1,L}$ or $\Gamma_{4,L}$. This means the result $A_{1,L}=A_{4,L}=0$ is not due to cancellation. 
\end{remark}

\begin{proof}
(i) This is equivalent to showing that the unstable subspace $\mathbb E_-^u(-\infty,\lambda)$ of
$$
	A_-(\lambda) = \begin{pmatrix} 0_{n} & D^{-1} \\ -\lambda I_{n} +  V_- & 0_{n} \end{pmatrix}
$$
has trivial intersection with $\mathcal D$ for any $\lambda \leq 0$. Let $U_k$ denote an eigenvector for $V_-$ with eigenvalue $k$. Since $V_-$ is symmetric, the collection $\{U_k\}_{k\in\Sp(V_-)}$ forms a basis for $\R^n$. From Hypothesis (H2) we have $k>0$, hence $k-\lambda > 0$. Then $\nu_\pm = \pm\sqrt{k-\lambda}$ is an eigenvalue for $A_-(\lambda)$ with eigenvector 
$$
V_k^\pm(\lambda) = \begin{pmatrix} U_k \\ \pm\sqrt{k-\lambda}DU_k\end{pmatrix}.
$$
Then $\text{span}\{V_k^\pm\} \cap\mathcal D = \{0\}$ since $\{U_k\}$ forms a basis of $\R^n$. 

(ii) Let $H_{(0)s}=-D\frac{d^2}{dx^2}$ with $\dom(H_{(0)s})=\dom(H_{s})$, as defined in \eqref{defLD}. By reflecting and translating, for any $s\in\bbR$ the spectrum of $H_{(0)s}$ is seen to be the same as $\Sp(-D\frac{\mathrm d^2}{\mathrm dx^2})$ on $L^2([0,+\infty);\R^n)$, which is well known to be $[0,+\infty)$; see, e.g., \cite[Theorem 6.33]{Wbook}.

By Theorem \ref{KatoTh} we infer
\begin{equation}\label{KatTh}
\dist\big(\Sp(H_{s}),\, \Sp(H_{(0)s})\big)\le\|V\|_{\cB(L^2((-\infty,s];\R^n))}\le\|V\|_\infty
\end{equation}
where $\|V\|_\infty$ is s-independent, hence $\Sp(H_{s})\subset[-\|V\|_\infty,+\infty)$. By Proposition \ref{cor:mult}, $s$ is a $\lambda$-conjugate point if and only if $\lambda$ is an eigenvalue of $H_{s}$. This means there are no $\lambda$-conjugate points satisfying $\lambda < -\|V\|_\infty$, and so $A_{4,L} = 0$.
\end{proof}

\section{Counting eigenvalues for $H_L$ via the Maslov index}\label{mmi}

We are now ready to state our first main result, Theorem~\ref{th:mas2}, which relates the Morse and Maslov indices of $H_L$.
\begin{theorem} \label{th:mas2}
Consider the operator $H_L$ defined in \eqref{ShEqL}. Fix $L<+\infty$ and $\lambda_\infty>0$ as in Lemma~\ref{lem:llss}, and let $A_{i,L} = \mi(\Gamma_{i,L},\mathcal D)$. Then the following assertions hold.
\renewcommand{\labelenumi}{{(\roman{enumi})}}
\begin{enumerate}
\item The Maslov index of the curve $\Gamma_L$ is zero.
\item $A_{3,L} = -A_{2,L}$.
\item $A_{3,L}\le0$ and $|A_{3,L}|$ is equal to the number of nonpositive eigenvalues for $H_L$, counting multiplicities:
 $$|A_{3,L}|=\mo\big(H_L\big)+\dim\ker\big(H_L\big).$$
\item $A_{2,L}\ge0$ and $A_{2,L}$ is equal to the number of the conjugate points in $(-\infty,L]$, counting multiplicities; see Definition~\ref{def:cpev}(1).
\item The Morse index $\mo(H_L)$ is equal to the number of conjugate points in $(-\infty,L)$, counting multiplicities.
\end{enumerate}
\end{theorem}

\begin{remark}\label{rem:sign}
In Theorem~\ref{th:mas2} and the lemmas below, the assumption $0\not\in\Sp(H_L)$ is not required. From Proposition \ref{cor:mult}, we see that $0 \in \Sp(H_L)$ precisely when $s=L$ is a conjugate point. In Figure \ref{F1}, this corresponds to a crossing at the top right corner. Such a crossing will contribute equally to both $A_{2,L}$ and $|A_{3,L}|$, so that its net effect on $\mi(\Gamma_L, \mathcal D)$ is zero. This cancellation explains why these endpoint contributions are present in (iii) and (iv) but not in (v). That these terms appear at all is a consequence of our choice of symplectic form
\begin{equation*}
	\Omega = \begin{pmatrix} 0_n & -I_n \\ I_n & 0_n \end{pmatrix},
\end{equation*}
as opposed to $-\Omega$, and the sign conventions in our definition of the Maslov index; see Definition \ref{def:maslov} and Remark \ref{rem:endpoints} above, and also \cite[Section 6.1]{CJLS}. The monotonicity computations in Lemmas \ref{lem:cross} and \ref{lem:scross} are also affected by this choice of symplectic form; much of the existing literature uses $-\Omega$, with respect to which these paths are negative definite.
\end{remark}

The proof of Theorem~\ref{th:mas2} relies on the monotonicity of the Maslov index with respect to the parameters $\lambda$ and $s$, which we establish in Lemmas~\ref{lem:cross} and \ref{lem:scross}, respectively. We begin with $\lambda$, following the strategy of \cite[Lemma 4.7]{DJ11}.

\begin{lemma}\label{lem:cross}
For any fixed $s\in(-\infty,L]$, the path $\lambda \mapsto \Phi_s(Y_\lambda)$ is positive definite. In particular, $A_{3,L}\le0$ and
\[
	|A_{3,L}| = \sum_{\lambda \leq 0} \dim \big( \Phi_L(Y_\lambda) \cap \mathcal D \big).
\]
\end{lemma}

\begin{proof}
Let $\lambda_0\in[-\lambda_\infty,0]$ be a crossing, so that $\Phi_s(Y_{\lambda_0})\cap \mathcal D\neq\{0\}$. Let $W$ be a subspace in $\R^{2n}$ transversal to $\Phi_s(Y_{\lambda_0})$.
Then $W$ is transversal to $\Phi_s(Y_\lambda)$ for all $\lambda\in[\lambda_0-\varepsilon,\lambda_0+\varepsilon]$ for $\varepsilon>0$ small enough. Thus, there exists a smooth family of matrices, $\phi(\lambda)$, for $\lambda\in[\lambda_0-\varepsilon,\lambda_0+\varepsilon]$, viewed as operators $\phi(\lambda)\colon \Phi_s(Y_{\lambda_0})\to W$, such that $\Phi_s(Y_\lambda)$ is the graph of $\phi(\lambda)$. Fix any nonzero ${\rm v}\in\Phi_s(Y_{\lambda_0})\cap \mathcal D$ and consider the curve $v(\lambda)={\rm v}+\phi(\lambda){\rm v}\in\Phi_s(Y_\lambda)$ for $\lambda\in[\lambda_0-\varepsilon,\lambda_0+\varepsilon]$ with $v(\lambda_0)={\rm v}$. By the definition of $Y_\lambda$, there is a family of solutions
$\bp(x;\lambda)$ of \eqref{ShEqp} such that
$v(\lambda)=\Phi_s\big(\bp(x;\lambda)\big)$. We claim that
\begin{equation}\label{negdef}
\omega\big(v(\lambda_0),\,\frac{\partial v}{\partial\lambda}(\lambda_0)\big)>0.
\end{equation}
Assuming the claim, we finish the proof as follows: Since for each  nonzero ${\rm v}\in\Phi_s(Y_{\lambda_0})\cap \mathcal D$ the crossing form $Q_\cM$ satisfies
\begin{align*}
Q_\cM({\rm v},{\rm v})&=\frac{d}{d\lambda}\Big|_{\lambda=\lambda_0}\omega({\rm v},\phi(\lambda){\rm v})=
\frac{d}{d\lambda}\Big|_{\lambda=\lambda_0}\omega({\rm v},{\rm v}+\phi(\lambda){\rm v})\\&=\omega\big(v(\lambda_0),\,\frac{\partial v}{\partial\lambda}(\lambda_0)\big)>0,
\end{align*}
the form is positive definite. Thus, the crossing $\lambda_0\in[-\lambda_\infty,0]$ is positive. In particular, taking into account that the path $\Gamma_{3,L}=\big\{\Phi_L(Y_{\lambda})\big\}_{\lambda=0}^{-\lambda_\infty}$ is parametrized by the parameter $\lambda$ {\em decreasing} from $0$ to $-\lambda_\infty$, each crossing $\lambda_0$ along $\Gamma_{3,L}$ is negative so that $A_{3,L} \le 0$.

Starting the proof of claim \eqref{negdef}, for the solution $\bp=\bp(x;\lambda)$ we compute the $\lambda$-derivative (for brevity, denoted below by dot) in equation \eqref{ShEqp}, and obtain the equation
\begin{equation}\label{lder}
\dot{\bp}'(x;\lambda)\big|_{\lambda=\lambda_0}=A(x;\lambda_0)\dot{\bp}(x;\lambda_0)+\sigma_0\bp(x;\lambda_0);
\end{equation}
where $\sigma_0=\begin{pmatrix}0_n&0_n\\-I_n&0_n\end{pmatrix}$. We also recall the definition $\Omega = \begin{pmatrix} 0_n&-I_n\\I_n&0_n\end{pmatrix}$ from \eqref{eq:Omega}, and will use the fact that $\Omega^\top = -\Omega$ and $p_1\in L^2((-\infty,L];\R^n)$. Then
\begin{align*}
\omega\big(v(\lambda_0),\,\frac{\partial v}{\partial\lambda}(\lambda_0)\big)&=\omega(\bp(s;\lambda),\,\dot\bp(s;\lambda))\big|_{\lambda=\lambda_0}=\langle\bp(s;\lambda),\,\Omega\dot\bp(s;\lambda)\rangle\big|_{\lambda=\lambda_0}\\
&=-\langle\Omega\bp(s;\lambda),\, \dot\bp(s;\lambda)\rangle\big|_{\lambda=\lambda_0}\\
&=-\int_{-\infty}^s \frac d{dx}\langle\Omega\bp(x;\lambda),\, \dot\bp(x;\lambda)\rangle\big|_{\lambda=\lambda_0} \,dx\\
&=-\int_{-\infty}^s\langle \Omega\bp(x;\lambda), \,\dot\bp'(x;\lambda)\rangle\big|_{\lambda=\lambda_0}\,dx\\
&\quad-\int_{-\infty}^s\langle \Omega\bp'(x;\lambda),\,\dot\bp(x;\lambda)\rangle\big|_{\lambda=\lambda_0}\,dx\\
&=-\underbrace{\int_{-\infty}^s\langle  \Omega\bp(x;\lambda_0), \,A(x,\lambda)\dot{\bp}(x;\lambda)\big|_{\lambda=\lambda_0}+\sigma_0\bp(x;\lambda_0)\rangle\,dx}_{\eqref{lder}} \\
&\quad-\underbrace{\int_{-\infty}^s\langle\Omega A(x,\lambda)\bp(x;\lambda_0),\,\dot\bp(x;\lambda)\big|_{\lambda=\lambda_0}\rangle\,dx}_{ \eqref{ShEqp}}\\
&=\int_{-\infty}^s\langle \left[\big(\Omega A(x,\lambda)\big)^\top-\Omega A(x,\lambda)\right]\bp(x;\lambda_0),\, \dot{\bp}(x;\lambda)\big|_{\lambda=\lambda_0}\rangle\,dx\\
&\quad-\int_{-\infty}^s\langle \Omega\bp(x;\lambda_0),\,\sigma_0\bp(x;\lambda_0)\rangle\,dx\\
&=-\underbrace{\int_{-\infty}^s\langle \Omega\bp(x;\lambda_0),\,\sigma_0\bp(x;\lambda_0) \rangle\,dx}_{\eqref{symm}}\\
&=\int_{-\infty}^s\langle \bp(x;\lambda_0),\, \Omega\sigma_0\bp(x;\lambda_0)\rangle\,dx\\
&=\int_{-\infty}^s\|p_1(x;\lambda_0)\|^2_{\R^{n}}\,dx>0
\end{align*}
thus completing the proof of \eqref{negdef} and the lemma.
\end{proof}

We will now establish  monotonicity of the Maslov index with respect to the parameter $s$.

\begin{lemma}\label{lem:scross}
For any fixed $\lambda\in[-\lambda_\infty,0]$, the path $s \mapsto \Phi_s(Y_\lambda)$ is positive. In particular, $A_{2,L} \geq 0$ and
\[
	A_{2,L} = \sum_{s \leq L} \dim \big( \Phi_s(Y_{0}) \cap \mathcal D \big).
\]
\end{lemma}

\begin{proof}

Let $s_\ast\in(-\infty,L]$ be a crossing, so that $\Phi_{s_\ast}(Y_\lambda)\cap \mathcal D\neq\{0\}$. Let $W$ be a subspace in $\R^{2n}$ transversal to $\Phi_{s_\ast}(Y_\lambda)$.
Then $W$ is transversal to $\Phi_s(Y_\lambda)$ for all $s\in[s_\ast-\varepsilon,s_\ast+\varepsilon]$ for $\varepsilon>0$ small enough. Thus, there exists a smooth family of matrices, $\phi(s)$, for $s\in[s_\ast-\varepsilon,s_\ast+\varepsilon]$, viewed as operators $\phi(s)\colon \Phi_{s_\ast}(Y_{\lambda})\to W$, such that $\Phi_s(Y_\lambda)$ is the graph of $\phi(s)$. Fix any nonzero ${\rm v}\in\Phi_{s_\ast}(Y_{\lambda})\cap \mathcal D$ and consider the curve $v(s)={\rm v}+\phi(s){\rm v}\in\Phi_s(Y_\lambda)$ for $s\in[s_\ast-\varepsilon, s_\ast+\varepsilon]$ with $v(s_\ast)={\rm v}$. By the definition of $Y_\lambda$, there is a family of solutions
$\bp(x;s)$ of \eqref{ShEqp} such that
$v(s)=\Phi_s \big(\bp(x;s)\big)$. The notation $\bp(x;s)$ denotes the fact that $\bp(x;s)\in Y_\lambda$ depends on $s$ as a parameter and is a function of $x$; then $\Phi_s\big(\bp(x;s)\big) = \bp(s;s)$ is the solution $\bp(x;s)$ evaluated at $x=s$. Denoting by dot the derivative with respect to the variable $s$ so that
\begin{align*}
\dot{v}(s)&=\big(\bp'(x;s)+\dot{\bp}(x;s)
  \big)\big|_{x=s}.
\end{align*}
 we claim that
\begin{equation}\label{snegdef}
\omega\big(v(s_\ast),\,\dot{v}(s_\ast)\big)>0.
\end{equation}

Assuming the claim, we finish the proof as follows: Since for each  nonzero ${\rm v}\in\Phi_{s_\ast}(Y_{\lambda})\cap \mathcal D$ the crossing form $Q_\cM$ satisfies
\begin{align}\lb{crossing form}
Q_\cM({\rm v},{\rm v})&=\frac{d}{ds}\Big|_{s=s_\ast}\omega({\rm v},\phi(s){\rm v})=
\frac{d}{ds}\Big|_{s=s_\ast}\omega({\rm v},{\rm v}+\phi(s){\rm v})\\&=\omega\big(v(s_\ast),\,\dot{v}(s_\ast)\big)>0,
\end{align}
the form is positive definite. Thus each crossing $s_\ast\in(-\infty,L]$ is positive, so $A_{2,L}\ge0$. 

Starting the proof of claim \eqref{snegdef}, we remark that $s$-derivatives of the solutions $\bp(\cdot,s)$ of \eqref{ShEqp} satisfy the differential equations
\begin{equation}\label{eqpwdot}
\dot{\bp}'(x)=A(x,\lambda)\dot{\bp}(x).
\end{equation}
Using the definition of $\omega$, we split the expression for $\omega\big(v(s),\,\dot{v}(s)\big)$ as follows:
\begin{align*}
\omega (v(s), \dot{v}(s))=&\langle  v(s),\Omega\dot{v}(s)\rangle\\
=&
\underbrace{\langle \bp(s;s),\, \Omega\bp'(s;s)\rangle}_{\alpha_{_1}(s)}+
\underbrace{\langle \bp(s;s),\, \Omega\dot{\bp}(s;s)\rangle}_{\alpha_{_2}(s)}.
\end{align*}
Using \eqref{eqpwdot} and rearranging terms, the expressions $\alpha_j$ are computed as follows:
\begin{align*}
	\alpha_1(s) &= \langle \bp(s;s), \Omega A(s,\lambda){\bp}(s;s)\rangle \\
	\alpha_2(s) &= \langle \bp(s;s), \Omega\dot{\bp}(s;s)\rangle \\
	&= \int_{-\infty}^s\frac{d}{dx}\big(\langle \bp(x;s), \Omega\dot{\bp}(x;s)\rangle\big)\,dx \\
	&= \int_{-\infty}^s\big(\langle A(x,\lambda) \bp(x;s), \Omega\dot{\bp}(x;s)\rangle
	+\langle \bp(x;s), \Omega A(x,\lambda)\dot{\bp}(x;s)\rangle\big)\,dx \\
	&= \int_{-\infty}^s\left< \big[-\Omega A(x,\lambda) +\big(\Omega A(x,\lambda)\big)^\top\big]\bp(x;s), \dot{\bp}(x;s)\right>\,dx
\\&=0.
\end{align*}
Thus, $\langle  v(s),\,\Omega\dot{v}(s)\rangle=\alpha_1(s)$. Using the condition $p_1(s_\ast;s_\ast)=0$, which holds since $s_*$ is a conjugate point, we conclude that
\[
	\omega\big(v(s_\ast),\,\dot{v}(s_\ast)\big)=\alpha_1(s_\ast)=\langle D^{-1}p_2(s_\ast,s_\ast),p_2(s_\ast,s_\ast)\rangle_{\R^{n}}>0.
\]
\end{proof}
We are now ready to prove Theorem~\ref{th:mas2}.
\begin{proof}[Proof of Theorem~\ref{th:mas2}]
Assertion (i) was already observed in \eqref{htpy}. Assertion (ii) follows from (i) and Lemma~\ref{lem:llss}. Assertion (iii) follows from Lemma~\ref{lem:cross} and Proposition~\ref{cor:mult}. Assertion (iv) follows from Lemma \ref{lem:scross} and the definition of a conjugate point. Finally, (v) follows from (ii)--(iv).
\end{proof}

We conclude this section with the observation that Theorem \ref{th:mas2} can be used to establish monotonicity of the eigenvalues of $H_L$.

\begin{remark}\label{rmk:mono}
Let $\lambda_1(L) \leq \lambda_2(L)\le\dots\le\lambda_{m}(L)\le0$ denote the nonpositive eigenvalues of the operator $H_L$. Since the essential spectrum of $H_L$ is strictly positive, the eigenvalues behave as the eigenvalues of a continuous family of finite-dimensional operators parametrized by $L$; see, e.g.,  \cite[Section II.5]{Kato}. A complete proof of this fact is given in \cite{LS16}. 

We claim that each $\lambda_j(L)$ is a strictly decreasing function of $L$. Applying Theorem \ref{th:mas2}(iii) to the operator $H_L - \lambda_j(L)$, we see that $|A_{3,L}|  \geq j$. Since $|A_{3,L}| = A_{2,L}$ is a nondecreasing function of $L$, we have $|A_{3,L+\delta}| \geq j$ for any $\delta > 0$, which implies that the first $j$ eigenvalues of $H_{L+\delta}$ are less than or equal to $\lambda_j(L)$. In particular, $\lambda_j(L+\delta) \leq \lambda_j(L)$, which shows that $\lambda_j$ is non-increasing. We complete the proof of the claim by noting that $\lambda_j$ cannot be locally constant, since the conjugate points are regular, and hence isolated, by Lemma \ref{lem:scross}. For an alternative proof we refer to \cite{LS16} where the derivatives $d\lambda_j/dL$ are computed in terms of the crossing form.

\end{remark}

\section{Counting eigenvalues for $H$ via the Maslov index}\label{ss:line}
We now show that Theorem~\ref{th:mas2} can be extended to count eigenvalues for problems on the whole real line. Recall the Schr\"odinger operator $H = -D \frac{d^2}{dx^2} + V(x)$ on $L^2(\bbR; \bbR^n)$, as defined in \eqref{ShEq}.


\begin{theorem}\label{main-line-thm} There exists $L_\infty \in \bbR$ such that for all $L>L_\infty$ the following hold:
\begin{itemize}
	\item[(i)] $\mo(H) = \mo(H_L)$;
	\item[(ii)] $L$ is not a conjugate point;
	\item[(iii)] $H_L$ is invertible.
\end{itemize}
In particular, the number of conjugate points is finite and independent of $L$, hence
\begin{align}
	\mo(H) =  \#\, \text{conjugate points in } (-\infty,+\infty).
\end{align}
\end{theorem}
Since $A_{2,L}$ converges as $L \to \infty$ to the number of conjugate points in $(-\infty,+\infty)$, we can interpret this number as the Maslov index for the whole-line problem.

In order to prove Theorem~\ref{main-line-thm} we use the following notation. For equation \eqref{ShEqp} 
\[
\bp' = A(x,\lambda) \bp, \quad A(x,\lambda)= \begin{pmatrix} 0_{n} & D^{-1} \\ -\lambda I_{n} +  V(x) & 0_{n} \end{pmatrix}
\]
we define the following subspaces
\begin{align*}
\mathbb E_+^s(x,\lambda) :& \text{ stable dichotomy subspace for $x\ge 0$}\\
\mathbb E_+^u(x,\lambda) :& \text{ unstable dichotomy subspace for $x\ge 0$}\\
\mathbb E_-^s(x,\lambda) :& \text{ stable dichotomy subspace for $x\le 0$}\\
\mathbb E_-^u(x,\lambda) :& \text{ unstable dichotomy subspace for $x\le 0$},
\end{align*}
recalling that $\mathbb E_-^u(s,\lambda) = \Phi_s(Y_\lambda)$.
We remark that although $\mathbb E_-^{s/u}(x,\lambda)$ and $\mathbb E_+^{s/u}(x,\lambda)$ are initially defined for $x\le0$ and $x\ge0$, respectively, the dichotomy subspaces can be propagated for all $x\in\R$. However, in general $(x,\lambda)\mapsto\mathbb E_-^u(x,\lambda)$ is not continuous on $(-\infty,+\infty]\times[-\lambda_\infty,0]$, and similarly for $(x,\lambda)\mapsto\mathbb E_+^s(x,\lambda)$ on $[-\infty,+\infty)\times[-\lambda_\infty,0]$. For more on this topic see \cite[Appendix]{HLS2}.

We also let $\mathbb E_\pm^{s/u} (\pm\infty,\lambda)$ denote the stable/unstable spectral subspaces of $A_\pm(\lambda)$.

\begin{lemma}\label{subspaceProperties}
The subspaces $\mathbb E_\pm^{s/u}(x,\lambda)$ have the following properties:
\begin{itemize}
\item[(i)] $\mathbb E_+^{s/u}(x,\lambda)\xrightarrow{x\to+\infty} \mathbb E_+^{s/u} (+\infty,\lambda)$ and $\mathbb E_-^{s/u}(x,\lambda)\xrightarrow{x\to-\infty}\mathbb E_-^{s/u} (-\infty,\lambda)$;

\item[(ii)] $\mathbb E_+^{s/u}(+\infty,\lambda)\cap\mathcal D = \{0\}$ and $\mathbb E_-^{s/u}(-\infty,\lambda)\cap\mathcal D = \{0\}$, with $\mathcal D$ as defined in \eqref{DefX1}. 
\end{itemize}
\end{lemma}
\begin{proof}
%
(i) This is a well-known fact about dichotomy subspaces. 

(ii) The result $\mathbb E_-^u(-\infty,\lambda)\cap\mathcal D = \{0\}$ was already shown in the proof of Lemma \ref{lem:llss}(ii); the proof is analogous for the remaining cases. 
\end{proof}

\begin{proof}[Proof of Theorem~\ref{main-line-thm}]
%
(i) We follow the proof of a similar result from \cite{SaS00}. By construction, $\mo(H)$ is the number of negative zeros of the Evans function; that is, the number of $\lambda < 0$ such that $$\mathbb E_+^s(0,\lambda)\wedge\mathbb E_-^u(0,\lambda)=0.$$ Here the wedge product $\wedge$ is defined to be $\det \big[ v^+_1(\lambda) \, \cdots \, v^+_n(\lambda) \ v^-_1(\lambda) \, \cdots \, v^-_n(\lambda) \big]$, where $\{v^\pm_j(\lambda)\}$ are analytic bases for $\mathbb E_+^s(0,\lambda)$ and $\mathbb E_-^u(0,\lambda)$. A change of basis will simply multiply the determinant by a nonvanishing analytic function, so the zeros and their multiplicities are well defined. Similarly, $\mo(H_L)$ is equal to the number of negative zeros of the function
\[
	\mathcal D\wedge \mathbb E_-^u(L,\lambda).
\]
We claim that $\mathbb E_+^s(0,\lambda)\wedge\mathbb E_-^u(0,\lambda)$ and $\mathcal D\wedge \mathbb E_-^u(L,\lambda)$ have the same number of negative zeros, counting multiplicity, for sufficiently large values of $L$. We remark that $\lambda=0$ can not be a cluster point of the set of negative zeros of the function $\mathcal D\wedge \mathbb E_-^u(L,\lambda)$ as $L\to\infty$ since the zeros are decreasing functions of $L$ by  Remark \ref{rmk:mono}.

Let $\phi(x_1,x_2;\lambda)$ denote the propagator of the non-autonomous differential equation $y' = A(x,\lambda)y$. Also denote $\mathcal D_L(\lambda) = \phi(0,L;\lambda)\mathcal D\wedge\mathbb E_-^u(0,\lambda)$ and $\mathcal D_\infty(\lambda) = \mathbb E_+^s(0,\lambda)\wedge\mathbb E_-^u(0,\lambda)$, and choose an analytic basis $\{v_j^+(\lambda)\}$ of $\mathbb E_+^s(0,\lambda)$. Note that $\mathcal D\oplus\mathbb E_+^u(+\infty,\lambda) = \R^{2n}$ because $\mathbb E_+^u(+\infty,\lambda)\cap\mathcal D = \{0\}$ (Lemma~\ref{subspaceProperties}(ii)). Also, note that $\mathcal D\wedge \mathbb E_-^u(L,\lambda)$ and $\mathcal D_L(\lambda)$ have the same zeros since $\mathbb E_-^u(L,\lambda) = \phi(L,0;\lambda)\mathbb E_-^u(0,\lambda)$, hence
\[
	\mathcal D_{L}(\lambda)= \phi(0,L;\lambda)\mathcal D\wedge\phi(0,L;\lambda)\mathbb E_-^u(L,\lambda) = [\det\phi(0,L;\lambda)]\mathcal D\wedge \mathbb E_-^u(L,\lambda).
\]

We now follow the proof of \cite[Lemma 4.3]{SaS00}. It is known that $\mathbb E_+^{s/u}(L,\lambda)\to\mathbb E_+^{s/u}(+\infty,\lambda)$ exponentially as $L\to+\infty$; see \cite[Thm. 1]{SaS00}. Then, as in \cite[Thm. 2]{SaS00}, there exist unique vectors $w_j^+(\lambda)\in\mathbb E_+^u(L,\lambda)$ such that $\mathcal D=\text{span}\{\phi(L,0;\lambda)v_j^+(\lambda)+w_j(\lambda):j=1,\ldots,n\}$ and
\[
	\phi(0,L;\lambda) \mathcal D = \text{span}\{v_j^+(\lambda)+\phi(0,L;\lambda)w_j^+(\lambda): j=1,\ldots,n\}.
\]
Thus $\phi(0,L;\lambda)\mathcal D$ and $\mathbb E_+^s(0,\lambda)$ are $\exp(-\sigma_+L)$-close, where $\sigma_+$ is the rate of exponential  decay of solutions at $+\infty$, and so the set of cluster points as $L\to\infty$ of the negative zeros of $\mathcal D_L(\lambda)$ is equal to the set of negative zeros of $\mathcal D_\infty(\lambda)$, counting multiplicity, by Rouch\'e's Theorem as in \cite[Rmk 4.3]{SaS00}, and by Remark \ref{rmk:mono} as mentioned above. 

(ii) Combining the result from (i) with Theorem~\ref{th:mas2}(v), we see that
\[
	\mo(H) = \mo(H_L) = \#\, \text{conjugate points in } (-\infty,L)
\]
for any $L > L_\infty$. Since $\mo(H)$ does not depend on $L$, we conclude that there are no conjugate points in $(L_\infty,+\infty)$, and so $\mo(H)$ is equal to the number of conjugate points in $\bbR$. 

(iii) This is equivalent to (ii), since $H_L$ fails to be invertible precisely when $L$ is a conjugate point.
\end{proof}

\begin{remark}
The fact that $\mo(H_L)$ converges to $\mo(H)$ as $L\to+\infty$ has also been shown in \cite{SW95}. Here we used an alternative proof based on the arguments from \cite{SaS00} (see also \cite{BL99}) where the spectrum of the operator on the full line was approximated by the spectra of the operator on $[-L, L]$. Unlike \cite{SW95}, this approach generalizes to operators which are not selfadjoint; this is relevant as recent progress has been made in defining Maslov-like indices for such problems; see, for instance, \cite{CJ2017}
\end{remark}

\begin{remark}
It is also possible to prove Theorem~\ref{main-line-thm}(ii) directly (without using part (i)) by showing that $\mathbb E_-^u(x,0)\cap\mathcal D = \{0\}$ for all $x\ge L_\infty$. If $0\not\in \Sp(H)$, then $\mathbb E_+^s(x,0)\cap\mathbb E_-^u(x,0) = \{0\}$ for all $x\ge0$, which in turn shows that $\mathbb E_+^s(x,0)\oplus\mathbb E_-^u(x,0) = \R^{2n}$ for all $x\ge0$. Since the unstable dichotomy subspace on $\R_+$ can be taken to be any subspace which complements $\mathbb E_+^s(x,0)$, we may therefore choose $E_+^u(x,0) = \mathbb E_-^u(x,0)$. Then by Lemma~\ref{subspaceProperties}(ii) we have $\mathbb E_+^u(x,0)\cap\mathcal D = \{0\}$ for sufficiently large $x$, as required. If $0 \in \Sp(H)$, we apply the previous argument to the perturbed operator $H + \varepsilon$, to find that $H_L + \varepsilon$ has no conjugate points for $L > L_\infty$. Since the eigenvalues are decreasing functions of $L$, by Remark \ref{rmk:mono}, the conjugate points will move down the $s$-axis  as $\varepsilon$ decreases to zero; see Figure \ref{F1}. Therefore $H_L$ has no conjugate points for $L > L_\infty$.
\end{remark}


\section{Instability of generic pulse solutions}\lb{ss:example}
Consider a reaction--diffusion system
\begin{align}\lb{eq:RD}
u_t = D u_{xx}+G(u),\quad u\in\R^n
\end{align}
where $G(u) = \nabla F(u)$ for some $C^2$ function $F\colon \R^n\to\R$, and $D$ is a diagonal diffusion matrix. We assume that there exists a stationary, spatially homogeneous solution $u_*(x,t) = u_0$ to \eqref{eq:RD}; without loss of generality we take $u_0=0$. We also assume that $k^2D - \nabla^2F(0) > 0$ for all $k \in \bbR$. This ensures that the spectrum of the linearization of \eqref{eq:RD} about $u_0$, which is given by
\[
\{\lambda \in \mathbb{R}: \det(k^2D + \lambda - \nabla^2 F(0))=0 \mbox{ for some } k \in \mathbb{R}\}, 
\]
lies in the open left half plane. We further suppose there is a stationary solution $\varphi^*(x,t)= \varphi^*(x)$ to \eqref{eq:RD} in $H^2(\bbR;\bbR^n)$. From the Sobolev embedding theorem and a bootstrap argument, $\varphi^*$ is at least $C^3$. Since $\nabla^2F(0)$ is nondegenerate, the invariant manifold theorem implies that $\varphi^*$ decays exponentially as $|x| \to \infty$. We thus call this a pulse, or pulse-type solution.
Due to our assumptions on the background state $u_0 = 0$, the essential spectrum of $\varphi^*$ is in the open left half plane. We claim that the pulse is unstable under a mild assumption which is generically satisfied. This generalizes the following classic result from Sturm--Liouville theory; c.f. \cite[\S 2.3.3.1]{KP}.
\begin{theorem}\lb{th:SL}
Suppose $\varphi^*(x)$ is a pulse-type solution to the scalar reaction--diffusion equation 
\begin{align}\lb{eq:1dRD}
u_t = u_{xx}+g(u), \quad u\in\R.
\end{align}
Then $\varphi^*(x)$ is unstable. 
\end{theorem}
We now show how the Maslov index can be used to establish the instability of any pulse-type solution to the system \eqref{eq:RD}. The eigenvalues for the linearization of \eqref{eq:RD} about $\varphi_*(x)$ solve
\[
	\lambda v = D\partial_x^2 v + \nabla^2 F(\varphi^*(x))v,
\]
and so it suffices to prove that the operator
\begin{align}\lb{eq:evalue}
	H = - D\partial_x^2 - \nabla^2 F(\varphi^*(x))
\end{align}
has at least one negative eigenvalue. We first show that \eqref{eq:evalue} satisfies Hypothesis~\ref{h1}, where $V(x) = -\nabla^2 F(\varphi^*(x))$.

\begin{enumerate}
\item[(H1)] Since $F$ is $C^2$, the matrix $-\nabla^2 F(\varphi^*(x))$ is symmetric and continuous in $x$.
\item[(H2)] Since $|\varphi^*(x)|\to0$ as $|x|\to\infty$, the limits $\lim_{x\to\pm\infty} V(x) = -\nabla^2 F(0)$ exist. Moreover, $V_\pm = -\nabla^2 F(0) > 0$ because it was assumed that $Dk^2 - \nabla^2F(0) > 0$ for all $k \in \mathbb{R}$ (and in particular $k=0$).
\item[(H3)] The functions $V(x) - V_\pm = -\nabla^2 F(\varphi^*(x)) + \nabla^2 F(0)$ are in $L^1(\R_\pm; R^{n\times n})$ since $\varphi^*(x)$ approaches 0 exponentially fast as $|x| \to \infty$.  
\end{enumerate}
Thus Theorem~\ref{main-line-thm} applies, so the existence of a conjugate point is enough to guarantee instability.

Writing the eigenvalue equation $Hv = \lambda v$ as a first order system, we obtain
\begin{align}\lb{eq:RDfirst}
	\bp' = A(x,\lambda)\bp, \quad A(x,\lambda)=\begin{pmatrix}
0_n & D^{-1} \\
-\lambda I_n-\nabla^2 F(\varphi^*(x))&0_n
\end{pmatrix},
\end{align}
where $p_1=v$, $p_2=Dv'$ and $\bp=(p_1,p_2)^\top \in \R^{2n}$, as in Section~\ref{s:intro}. By definition, conjugate points are intersections of $\Phi_s(Y_{0})$ with the Dirichlet subspace $\mathcal D$, where we recall from Section~\ref{s:intro} that
%
\[
	\Phi_s(Y_0) = \left\{\bp(s)  : \bp'=A(x,0)\bp \ \ \text{ and } \lim_{x\to-\infty} \bp(x) = 0 \right\}
\]
and $
\mathcal D = \{ (p_1,p_2)^\top\in \R^{2n}\,\big|\, p_1=0 \}.
$
Letting $\mathbb E^u_-(x,0)$ denote the unstable subspace for \eqref{eq:RDfirst} with $\lambda=0$ coming from $x= -\infty$, it is clear that $\Phi_s(Y_{0}) = \mathbb E^u_-(s,0)$ for any $s<+\infty$. Thus our task is to find a basis for $\mathbb E^u_-(x,0)$, and then count its intersections with $\mathcal D$. Since $\mathbb E^u_-(s,0)$ is Lagrangian, by Theorem \ref{th:lagrange}, it is $n$-dimensional.

Differentiating \eqref{eq:RD} with respect to $x$, we find that $(\varphi_x^*(x),\varphi_{xx}^*(x))^\top$ is a solution to \eqref{eq:RDfirst} with $\lambda=0$; thus $(\varphi_x^*(x),\varphi_{xx}^*(x))^\top\in\mathbb E^u_-(x,0)$. Let $(v_j(x),\partial_x v_j(x))^\top$, $j\in\{1,\ldots,n-1\}$, denote the remaining $n-1$ basis vectors for $\mathbb E_-^u(x,0)$, which are unknown. Denoting the $i$th component of $\varphi^*(x)$ by $\varphi_i^*(x)$ and the $i$th component of $v_j(x)$ by $v_{j,i}(x)$, we have
\[
\mathbb E^u_-(x,0) = \mbox{span}\left\{ \begin{pmatrix} \partial_x\varphi^*_1(x) \\\vdots\\ \partial_x\varphi^*_n(x)\\ \partial_{xx}\varphi^*_1(x) \\\vdots\\ \partial_{xx}\varphi^*_n(x) \end{pmatrix}, \begin{pmatrix} v_{1,1}(x)\\\vdots \\v_{1,n}(x)\\\partial_x v_{1,1}(x)\\\vdots\\\partial_x v_{1,n}(x)\end{pmatrix},\ldots, \begin{pmatrix} v_{n-1,1}(x)\\\vdots \\v_{n-1,n}(x)\\\partial_x v_{n-1,1}(x)\\\vdots\\\partial_x v_{n-1,n}(x)\end{pmatrix}\right\}.
\]
Then finding intersections with $\mathcal D$ reduces to finding values of $s$ so that
\begin{align}\lb{intersection}
\det\begin{pmatrix}
\partial_x\varphi^*_1(s) & v_{1,1}(s) & \cdots & v_{n-1,1}(s)\\\vdots&\vdots&\ddots&\vdots\\ \partial_x\varphi^*_n(s) & v_{1,n}(s) & \cdots & v_{n-1,n}(s)
\end{pmatrix} = 0.
\end{align}
We emphasize that the monotonicity along $\Gamma_{2,L} = \{(\lambda,s)\in \{0\}\times[-\infty,L]\}$ (Lemma~\ref{lem:scross}) implies any such intersection must be non-degenerate.
Unlike the scalar case, it is not immediately clear that simply knowing $\varphi^*(x)$ is a pulse provides enough information to conclude that there exists an $s$ satisfying \eqref{intersection} since, in general, it is not obvious how to find the vectors $v_j(x)$. However, if we can find an $x_0$ so that all of the derivatives $\partial_x\varphi_i^*(x_0)$ are \emph{simultaneously} zero, then \eqref{intersection} is satisfied for $s=x_0$, regardless of the vectors $v_j(x_0)$. We will show that such an $x_0$ exists by showing that the original pulse solution $\varphi^*(x)$ is even-symmetric about some $x_0$. We make the following assumption.

\begin{hypothesis}\lb{h:unique}
Consider the first-order system of equations describing stationary solutions to \eqref{eq:RD} 
\begin{align}\lb{eq:firstRD}
\begin{split}
	u_x &= D^{-1} v \\
	v_x &= -G(u)
\end{split}
\end{align}
and let $\mathcal W^s(x)$ and $\mathcal W^u(x)$ denote the stable and unstable manifolds of $(u,v)^\top = 0$ associated with \eqref{eq:firstRD}. 
We assume that $(\varphi^*(x),\varphi_{x}^*(x))^\top$ is the unique solution, up to spatial translation, contained in the intersection $\mathcal W^s(x)\cap\mathcal W^u(x)$.
\end{hypothesis}

\begin{remark}
Since we assume that $\varphi^*(x)$ is a pulse solution to \eqref{eq:RD}, $\dim(\mathcal W^s(x)\cap\mathcal W^u(x))\ge1$. The assumption that this dimension is exactly equal to one is generic for the following reason. We append the $x$ direction so that the manifolds $\mathcal W^s(x)$ and $\mathcal W^u(x)$ are $n+1$ dimensional manifolds in a $2n+1$ dimensional ambient space. Then it is a well known fact of differential topology that the dimension of a transverse intersection of two manifolds $X$ and $Z$ in the ambient space $Y$ is given by
$$
\dim(X\cap Z) = \dim(X)+\dim(Z) - \dim(Y)
$$
(c.f. \cite[pg. 30]{GP}) which in our case gives $$\dim(\mathcal W^s(x)\cap\mathcal W^u(x)) = (n+1)+(n+1) - (2n+1) = 1.$$ 
\end{remark}

The even-symmetry of $\varphi^*(x)$ is a straightforward consequence of Hypothesis~\ref{h:unique} and the spatial-reversibility of \eqref{eq:RD}.
\begin{proposition}\lb{pp:symmetric}
Assume Hypothesis~\ref{h:unique}. Then there exists some $x_0\in\R$ so that $\varphi^*(x)$ is even-symmetric about $x=x_0$.
\end{proposition}

\begin{proof}
Equation \eqref{eq:RD} is reversible; i.e. if $u(x)$ is a solution, so is $u(-x)$. Since $\varphi^*(x)$ is a solution to \eqref{eq:RD}, so is $\varphi^*(-x)$ and, by the definition of a pulse, both solutions are contained in the intersection $\mathcal W^s(x)\cap\mathcal W^u(x)$. By Hypothesis~\ref{h:unique}, $\varphi^*(x)$ and $\varphi^*(-x)$ are the same up to spatial translations. This can only be true if $\varphi^*(x)$ is even-symmetric about some point $x_0$. More precisely, if $\varphi^*(x) = \varphi^*(-x + \delta)$ for all $x \in \mathbb{R}$ and some fixed $\delta$, then $\varphi^*(x_0 + x) = \varphi^*(x_0-x)$ for all $x \in \mathbb{R}$, where $x_0 = \delta/2$.
\end{proof}

We are now ready to prove our main result.

\begin{theorem}\lb{th:example}
Assume Hypothesis~\ref{h:unique}. Then $\varphi^*(x)$ is unstable. 
\end{theorem}
\begin{proof}
By Proposition~\ref{pp:symmetric}, $\varphi^*(x)$ is even-symmetric about some $x_0$; thus $$\partial_x\varphi^*(x)\Big|_{x=x_0} = 0$$ and so \eqref{intersection} is satisfied for $s=x_0$. 
\end{proof}

\begin{remark}
It was assumed as the start of the section that the essential spectrum of the linearized operator was stable; this allowed us to verify Hypothesis \ref{h1} and thus apply Theorem~\ref{main-line-thm} to conclude instability of the pulse. On the other hand, if $-\nabla^2F(0)$ has any negative eigenvalues, the pulse is unstable due to the essential spectrum.
\end{remark}

\begin{remark}
The proof of Theorem \ref{th:example} only requires the even-symmetry of $\varphi^*$ (which may be valid even if Hypothesis \ref{h:unique} does not hold, or cannot be verified). We choose to state the result in terms of Hypothesis \ref{h:unique}, rather than its consequence, Proposition \ref{pp:symmetric}, as it is less apparent that the latter condition is generically satisfied.
\end{remark}

\section*{Acknowledgments} The authors wish to thank the American Institute of Mathematics (AIM) in San Jose, CA, and their SQuaRE program, where much of this work was conducted. M.B. would like to thank Arnd Scheel for suggesting the potential for using the Maslov index to investigate the stability of Turing patterns. The work of M.B. was partly supported by US National Science Foundation (NSF) grant DMS--1411460. The work of C.J. was partly supported by US NSF grant DMS--1312906. Y.L. was supported by the US NSF grant DMS--1710989, by the Research Board and Research Council of the University of Missouri, and by the Simons Foundation.


\bibliographystyle{plain}
\bibliography{bcjlms}
 

\end{document}